\documentclass[amstex,12pt, amssymb]{article}

\usepackage{mathtext}
\usepackage[cp1251]{inputenc}
\usepackage[T2A]{fontenc}
\usepackage[dvips]{graphicx}
\usepackage{amsmath}
\usepackage{amssymb}
\usepackage{amsxtra}
\usepackage{latexsym}
\usepackage{ifthen}

\newcounter{lemma}[section]

\newcounter{corollary}[section]

\newcounter{remark}[section]

\newcounter{theorem}[section]

\newcounter{proposition}[section]

\newcounter{example}

\numberwithin{equation}{section}

\begin{document}

\markboth{V. DESYATKA, E.~SEVOST'YANOV}{\centerline{ON CARATHEODORY
THEOREMS ...}}

\def\cc{\setcounter{equation}{0}
\setcounter{figure}{0}\setcounter{table}{0}}

\overfullrule=0pt


\author{ZARINA KOVBA, EVGENY SEVOST'YANOV}

\title{
{\bf ON CARATHEODORY PRIME ENDS EXTENSION FOR UNCLOSED
ORLICZ-SOBOLEV CLASSES}}

\date{\today}
\maketitle

\begin{abstract}
We study problems related to continuous boundary extension of
mappings of Orlicz-Sobolev classes in terms of prime ends. The
results we obtain concern the case when the mappings are open,
discrete, but not closed (not preserving the boundary of a domain).
These results generalize the well-known results of Caratheodory on
boundary extension of conformal mappings.
\end{abstract}


\section{Introduction}

This paper is devoted to the study of mappings with bounded and
finite distortion, see, e.g., \cite{Cr}, \cite{MRV}, \cite{MRSY},
\cite{PSS}, \cite{RV}, \cite{SalSt} and~\cite{Va}. The well-known
Caratheodory theorem states that conformal mappings extend to the
boundaries of their domains in terms of prime ends (see, e.g.,
\cite[Theorem~9.4]{CL}). There is also a similar result of Nakki on
quasiconformal mappings, the content of which we present below.

\medskip
{\bf Theorem.}{\it\, Under a quasiconformal mapping $f$ of a
collared domain $D_0$ onto a domain $D,$ there exists a one-to-one
correspondence between the boundary points of $D_0$ and the prime
ends of $D.$ Moreover, the cluster set $C(f, b),$ $b\in \partial
D_0,$ coincides with the impression $I(P)$ of the corresponding
prime end $P$ of $D$ (see~\cite[Theorem~4.1]{Na}).}

\medskip Let us note many other results made in this direction, see
e.g.~\cite{Ad, Car, Cr, KPR, KR}. The goal of the paper is to
transfer these and some other results to the Sobolev and
Orlicz-Sobolev classes, which are not homeomorphisms and do not
preserve the boundary, but are open and discrete. We should note
that for homeomorphisms of Orlicz-Sobolev classes, similar results
were obtained in~\cite{KR}. In the case when the mappings are open,
discrete, and closed the above was done in~\cite[Ch.~5]{Sev$_4$}.
The results obtained below are in many ways similar to~\cite{DS} and
\cite{DSH}. We should clarify that the modulus distortion estimates
at the boundary points {\bf are not established} for boundary
non-preserving mappings in $W^{1, \varphi}_{\rm loc}(D).$ If
$\int\limits_{1}^{\infty}\left(\frac{t}{\varphi(t)}\right)^
{\frac{1}{n-2}}\,dt<\infty$ (the Calderon condition), all such
estimates assume either that $f$ is a homeomorphism, or $f$ is open,
discrete and closed; see, e.g., \cite{Sev$_4$}. In the latter case,
the problem of a boundary extension of Orlicz-Sobolev mappings
directly reduces to modulus estimates. At the same time, by virtue
of the above, this problem does not have such a simple connection
for boundary non-preserving mappings. Our goal is to find it and to
study the boundary behavior of corresponding mappings.

\medskip
Recall some definitions. Let $U$ be an open set in ${\Bbb R}^n.$ In
what follows, $C^k_0(U)$ denotes the space of functions $u:U
\rightarrow {\Bbb R} $ with a compact support in $U,$ having $k$
partial derivatives with respect to any variable that are continuous
in $U.$ Let $U$ be an open set, $U\subset{\Bbb R}^n,$ and let
$u:U\rightarrow {\Bbb R}$ be some function, $u \in L_{\rm
loc}^{\,1}(U).$ Assume that, there exists a function $v\in L_{\rm
loc}^{\,1}(U)$ such that
$$\int\limits_U \frac{\partial \varphi}{\partial x_i}(x)u(x)\,dm(x)=
-\int\limits_U \varphi(x)v(x)\,dm(x)$$
for any function $\varphi\in C_1^{\,0}(U).$ Then we say that the
function $v$ is a weak derivative of the first order of~$u$ with
respect to $x_i$ and denoted by the $\frac{\partial u}{\partial
x_i}(x):= v.$ We write $u\in W_{\rm loc}^{1,1}(U)$ if $u$ has weak
derivatives of the first order with respect to each of the variables
in $U$ and $\frac{\partial u}{\partial x_i}$ are locally integrable
in~$U.$

\medskip
A mapping $f:U\rightarrow {\Bbb R}^n$ belongs to the Sobolev class
$W_{\rm loc}^{1,1}(U),$ we write $f \in W^{1,1}_{\rm loc}(U),$ if
all coordinate functions of $f=(f_1,\ldots, f_n)$ have weak partial
derivatives of the first order, which are locally integrable in $U.$
We write $f\in W^{1, p}_{\rm loc}(U),$ $p\geqslant 1,$ if all
coordinate functions $f_i,$ $1\leqslant i\leqslant n,$ of
$f=(f_1,\ldots, f_n)$ have weak partial derivatives of the first
order, which are locally integrable in $U$ to the degree $p.$

\medskip
Let $\varphi:{\Bbb R}^+\rightarrow{\Bbb R}^+$ be a measurable
function. The {\it Orlicz-Sobolev class} $W^{1,\varphi}_{\rm
loc}(D)$ is the class of all locally integrable mappings $f$ with
the first distributional derivatives whose gradient $\nabla f$
belongs locally in $D$ to the Orlicz class. By the definition,
$W^{1,\varphi}_{\rm loc}\subset W^{1,1}_{\rm loc}$. For the case
when $\varphi(t)=t^p$, $p\geqslant 1,$ we write as usual $f\in
W^{1,p}_{\rm loc}.$

\medskip
In more detail, we write $f\in W^{1,\varphi}_{\rm loc}(D),$ if
$f_i\in W^{1,1}_{\rm loc}(D)$ and
$$
\int\limits_{G}\varphi\left(|\nabla f(x)|\right)\,dm(x)<\infty
$$
for any domain $G\subset D$ with $\overline{G}\subset D,$ $|\nabla
f(x)|=\sqrt{\sum\limits_{i,j}\left(\frac{\partial f_i}{\partial
x_j}\right)^2}.$

\medskip
Recall that a mapping $f:D\rightarrow {\Bbb R}^n$ is called {\it
discrete} if the pre-image $\{f^{-1}\left(y\right)\}$ of each point
$y\,\in\,{\Bbb R}^n$ consists of isolated points, and {\it is open}
if the image of any open set $U\subset D$ is an open set in ${\Bbb
R}^n.$

\medskip
Given sets $E,$ $F\subset\overline{{\Bbb R}^n}$ and a domain
$D\subset {\Bbb R}^n$ we denote by $\Gamma(E,F,D)$ a family of all
paths $\gamma:[a,b]\rightarrow \overline{{\Bbb R}^n}$ such that
$\gamma(a)\in E,\gamma(b)\in\,F $ and $\gamma(t)\in D$ for $t \in
(a, b).$ Given a mapping $f:D\rightarrow {\Bbb R}^n$, we denote
\begin{equation}\label{eq1_A_4} C(f, x):=\{y\in \overline{{\Bbb
R}^n}:\exists\,x_k\in D: x_k\rightarrow x, f(x_k) \rightarrow y,
k\rightarrow\infty\}
\end{equation}
and
\begin{equation}\label{eq1_A_5} C(f, \partial
D)=\bigcup\limits_{x\in \partial D}C(f, x)\,.
\end{equation}
In what follows, ${\rm Int\,}A$ denotes the set of inner points of
the set $A\subset \overline{{\Bbb R}^n}.$ Recall that the set
$U\subset\overline{{\Bbb R}^n}$ is neighborhood of the point $z_0,$
if $z_0\in {\rm Int\,}A.$

\medskip
We say that a function ${\varphi}:D\rightarrow{\Bbb R}$ has a {\it
finite mean oscillation} at a point $x_0\in D,$ write $\varphi\in
FMO(x_0),$ if
\begin{equation}\label{eq29*!}
\limsup\limits_{\varepsilon\rightarrow
0}\frac{1}{\Omega_n\varepsilon^n}\int\limits_{B( x_0,\,\varepsilon)}
|{\varphi}(x)-\overline{{\varphi}}_{\varepsilon}|\ dm(x)<\infty\,,
\end{equation}
where
$$\overline{{\varphi}}_{\varepsilon}=\frac{1}
{\Omega_n\varepsilon^n}\int\limits_{B(x_0,\,\varepsilon)}
{\varphi}(x) dm(x)\,.$$
Let $Q:{\Bbb R}^n\rightarrow [0,\infty]$ be a Lebesgue measurable
function. We set
$$Q^{\,\prime}(x)=\left\{
\begin{array}{rr}
Q(x), &   Q(x)\geqslant 1\,, \\
1,  &  Q(x)<1\,.
\end{array}
\right.$$ Denote by $q^{\,\prime}_{x_0}$ the mean value of
$Q^{\,\prime}(x)$ over the sphere $|x-x_0|=r$, that means,
\begin{equation}\label{eq32*B}
q^{\,\prime}_{x_0}(r):=\frac{1}{\omega_{n-1}r^{n-1}}
\int\limits_{|x-x_0|=r}Q^{\,\prime}(x)\,d{\mathcal H}^{n-1}\,.
\end{equation}
Note that, using the inversion $\psi(x)=\frac{x}{|x|^2},$ we may
give the definition of $FMO$ as well as for $x_0=\infty.$ Similarly,
we may define~(\ref{eq32*B}) for $x_0=\infty.$

\medskip
A Borel function $\rho:{\Bbb R}^n\,\rightarrow [0,\infty] $ is
called {\it admissible} for the family $\Gamma$ of paths $\gamma$ in
${\Bbb R}^n,$ if the relation
\begin{equation*}\label{eq1.4}
\int\limits_{\gamma}\rho (x)\, |dx|\geqslant 1
\end{equation*}
holds for all (locally rectifiable) paths $ \gamma \in \Gamma.$ In
this case, we write: $\rho \in {\rm adm} \,\Gamma .$ Let $p\geqslant
1,$ then {\it $p$-modulus} of $\Gamma $ is defined by the equality
\begin{equation*}\label{eq1.3gl0}
M_p(\Gamma)=\inf\limits_{\rho \in \,{\rm adm}\,\Gamma}
\int\limits_{{\Bbb R}^n} \rho^p (x)\,dm(x)\,.
\end{equation*}

\medskip
We say that the boundary $\partial D$ of a~domain $D$ in ${\Bbb
R}^n,$ $n\geqslant 2,$ is {\it strongly accessible at a~point
$x_0\in
\partial D$ with respect to the $p$-modulus} if for each neighborhood
$U$ of $x_0$ there
exist a~compact set $E\subset D$, a~neighborhood $V\subset U$ of
$x_0$ and $\delta>0$ such that
\begin{equation}
\label{eq1.3_a} M_p(\Gamma(E,F, D))\geqslant \delta
\end{equation}
for each continuum $F$ in~$D$ that intersects $\partial U$ and
$\partial V$. When $p=n$, we will usually drop the prefix in the
``$p$-modulus'' when speaking about~(\ref{eq1.3_a}).

\medskip
Assume that, a mapping $f$ has partial derivatives almost everywhere
in $D.$ In this case, we set
\begin{gather}l\left(f^{\,\prime}(x)\right)\,=\,\min\limits_{h\in {\Bbb R}^n
\backslash \{0\}} \frac {|f^{\,\prime}(x)h|}{|h|}\,, \nonumber\\
\label{eq5_a} \Vert f^{\,\prime}(x)\Vert\,=\,\max\limits_{h\in {\Bbb
R}^n \backslash \{0\}} \frac {|f^{\,\prime}(x)h|}{|h|}\,,\\
J(x,f)=\det f^{\,\prime}(x)\,.\nonumber\end{gather}
and define for any $x\in D$ and $\alpha\geqslant 1$
\begin{equation}\label{eq0.1.1A}
K_{I, \alpha}(x,f)\quad =\quad\left\{
\begin{array}{rr}
\frac{|J(x,f)|}{{l\left(f^{\,\prime}(x)\right)}^{\,\alpha}}, & J(x,f)\ne 0,\\
1,  &  f^{\,\prime}(x)=0, \\
\infty, & {\rm otherwise}
\end{array}
\right.\,,\end{equation}
The next definitions due to Caratheodory~\cite{Car}; cf.~\cite{KR}.
Let $\omega$ be an open set in ${\Bbb R}^k$, $k=1,\ldots,n-1$. A
continuous mapping $\sigma\colon\omega\rightarrow{\Bbb R}^n$ is
called a {\it $k$-dimensional surface} in ${\Bbb R}^n$. A {\it
surface} is an arbitrary $(n-1)$-dimensional surface $\sigma$ in
${\Bbb R}^n.$ A surface $\sigma$ is called {\it a Jordan surface},
if $\sigma(x)\ne\sigma(y)$ for $x\ne y$. In the following, we use
$\sigma$ instead of $\sigma(\omega)\subset {\Bbb R}^n,$
$\overline{\sigma}$ instead of $\overline{\sigma(\omega)}$ and
$\partial\sigma$ instead of
$\overline{\sigma(\omega)}\setminus\sigma(\omega).$ A Jordan surface
$\sigma\colon\omega\rightarrow D$ is called a {\it cut} of $D$, if
$\sigma$ separates $D,$ that is $D\setminus \sigma$ has more than
one component, $\partial\sigma\cap D=\varnothing$ and
$\partial\sigma\cap\partial D\ne\varnothing$.

A sequence of cuts $\sigma_1,\sigma_2,\ldots,\sigma_m,\ldots$ in $D$
is called {\it a chain}, if:

(i) the set $\sigma_{m+1}$ is contained in exactly one component
$d_m$ of the set $D\setminus \sigma_m,$ wherein $\sigma_{m-1}\subset
D\setminus (\sigma_m\cup d_m)$; (ii)
$\bigcap\limits_{m=1}^{\infty}\,d_m=\varnothing.$

Two chains of cuts  $\{\sigma_m\}$ and $\{\sigma_k^{\,\prime}\}$ are
called {\it equivalent}, if for each $m=1,2,\ldots$ the domain $d_m$
contains all the domains $d_k^{\,\prime},$ except for a finite
number, and for each $k=1,2,\ldots$ the domain $d_k^{\,\prime}$ also
contains all domains $d_m,$ except for a finite number.

The {\it end} of the domain $D$ is the class of equivalent chains of
cuts in $D$. Let $K$ be the end of $D$ in ${\Bbb R}^n$, then the set
$I(K)=\bigcap\limits_{m=1}\limits^{\infty}\overline{d_m}$ is called
{\it the impression of} $K$. Following~\cite{Na}, we say that the
end $K$ is {\it a prime end}, if $K$ contains a chain of cuts
$\{\sigma_m\}$ such that
$$\lim\limits_{m\rightarrow\infty}M(\Gamma(C, \sigma_m, D))=0$$
for some continuum $C$ in $D.$ The following notation is used: the
set of prime ends corresponding to the domain $D$ is denoted by
$E_D,$ and the completion of the domain $D$ by its prime ends is
denoted by $\overline{D}_P.$

\medskip
Consider the following definition, which goes back to
N\"akki~\cite{Na}, cf.~\cite{KR}. The boundary of a domain $D$ in
${\Bbb R}^n$ is said to be {\it locally quasiconformal} if every
$x_0\in\partial D$ has a neighborhood $U$ that admits a
quasiconformal mapping $\varphi$ onto the unit ball ${\Bbb
B}^n\subset{\Bbb R}^n$ such that $\varphi(\partial D\cap U)$ is the
intersection of ${\Bbb B}^n$ and a coordinate hyperplane. The
sequence of cuts $\sigma_m,$ $m=1,2,\ldots ,$ is called {\it
regular,} if
$\overline{\sigma_m}\cap\overline{\sigma_{m+1}}=\varnothing$ for
$m\in {\Bbb N}$ and, in addition, $d(\sigma_{m})\rightarrow 0$ as
$m\rightarrow\infty.$ If the end $K$ contains at least one regular
chain, then $K$ will be called {\it regular}. We say that a bounded
domain $D$ in ${\Bbb R}^n$ is {\it regular}, if $D$ can be
quasiconformally mapped onto a domain with a locally quasiconformal
boundary whose closure is a compactum in ${\Bbb R}^n,$ and, besides
that, every prime end in $D$ is regular. Note that, the space
$\overline{D}_P=D\cup E_D$ is metric, which can be demonstrated as
follows. If $g:D_0\rightarrow D$ is a quasiconformal mapping of a
domain $D_0$ with a locally quasiconformal boundary onto some domain
$D,$ then for $x, y\in \overline{D}_P$ we put:
\begin{equation}\label{eq5M}
\rho(x, y):=|g^{\,-1}(x)-g^{\,-1}(y)|\,,
\end{equation}
where the element $g^{\,-1}(x),$ $x\in E_D,$ is to be understood as
some (single) boundary point of the domain $D_0.$ The specified
boundary point is unique and well-defined by~\cite[Theorem~2.1,
Remark~2.1]{IS}, cf.~\cite[Theorem~4.1]{Na}. It is easy to verify
that~$\rho$ in~(\ref{eq5M}) is a metric on $\overline{D}_P.$ If
$g_*$ is another quasiconformal mapping of a domain $D_*$ with
locally quasiconformal boundary onto $D$, then the corresponding
metric
$\rho_*(p_1,p_2)=|{\widetilde{g_*}}^{-1}(p_1)-{\widetilde{g_*}}^{-1}(p_2)|$
generates the same convergence and, consequently, the same topology
in $\overline {D}_P$ as $\rho_0$ because $g_0\circ g_*^{-1}$ is a
quasiconformal mapping of $D_*$ onto $D_0$, which extends, by
Theorem 4.1 in~\cite{Na}, to a homeomorphism between $\overline
{D_*}$ and $\overline {D_0}$. In the sequel, this topology in
$\overline {D}_P$ will be called the {\it topology of prime ends};
the continuity of mappings $F\colon \overline
{D}_P\rightarrow\overline{D^{\,\prime}}_P$ will be understood
relative to this topology.

\medskip
Some analogues of the following results were established for
homeomorphisms in~\cite[Theorems 3, 4]{KR}. For open discrete and
closed mappings, see, e.g., \cite[Ch.~5]{Sev$_4$}. For mappings with
upper modulus inequalities that are not closed, see, for example,
\cite{DS}.

\medskip
\begin{theorem}\label{th3_1}
{\it\, Let $n-1<\alpha\leqslant n,$ let $D$ and $D^{\,\prime}$ be
bounded domains in ${\Bbb R}^n,$ $n\geqslant 3,$ let $Q:D\rightarrow
[0, \infty]$ be a Lebesgue measurable function and let
$\varphi:[0,\infty)\rightarrow[0,\infty)$ be an increasing function.
Let $D$ be a regular domain and let $f:D\rightarrow D^{\,\prime}$ be
an open discrete mapping in $W^{1,\varphi}_{\rm loc}(D),$
$f(D)=D^{\,\prime}.$ In addition, assume that $C(f,
\partial D)\subset E_*$ for some closed (in the topology of ${\Bbb R}^n$) set $E_*\subset
\overline{D^{\,\prime}}$ and $f^{\,-1}(E_*)=E$ for some closed (in
the topology of $D$) subset $E\subset D.$ In addition, assume that:

\medskip
1) the set $E$ is nowhere dense in $D$ and $D$ is finitely connected
on $E\cup \partial D,$ i.e., for any $z_0\in E\cup \partial D$ and
any neighborhood $\widetilde{U}$ of $z_0$ there is a neighborhood
$\widetilde{V}\subset \widetilde{U}$ of $z_0$ such that $(D\cap
\widetilde{V})\setminus E$ consists of finite number of components,

\medskip
2) for any $P\in E_D:=\overline{D}_P\setminus D$ and for every
neighborhood $U$ of $P$ in $\overline{D}_P$ there is a neighborhood
$V\subset U$ in $\overline{D}_P$ of $P$ such that $V\cap D$ is
connected and $(V\cap D)\setminus E$ consists at most of $m$
components, $1\leqslant m<\infty,$

\medskip
3) all components of the set $D^{\,\prime}\setminus E_*$ have a
strongly accessible boundary with respect to $\alpha$-modulus,

\medskip
4) the function $\varphi$ satisfies the following Calderon condition
\begin{equation}\label{eq1_A_10}
\int\limits_{1}^{\infty}\left(\frac{t}{\varphi(t)}\right)^
{\frac{1}{n-2}}\,dt<\infty\,.
\end{equation}
5) Assume that $K_{I, \alpha}(x, f)\leqslant Q(x)$ a.e. and, in
addition, there exists $M>0$ and some increasing convex function
$\Phi:\overline{{\Bbb R^{+}}}\rightarrow \overline{{\Bbb R^{+}}}$
such that
\begin{equation}\label{eq1D}
\int\limits_{D}\Phi(Q(x))\cdot\frac{dm(x)}{(1+|x|^2)^n}\leqslant M
\end{equation}
while
\begin{equation}\label{eq2} \int\limits_{\delta_0}^{\infty}
\frac{d\tau}{\tau\left[\Phi^{-1}(\tau)\right]^{\frac{1}{\alpha-1}}}=
\infty
\end{equation}
for some $\delta_0>\tau_0:=\Phi(0).$ Then $f$ has a continuous
extension
$\overline{f}:\overline{D}_P\rightarrow\overline{D^{\,\prime}},$
moreover, $\overline{f}(\overline{D}_P)=\overline{D^{\,\prime}}.$ }
\end{theorem}

\medskip
\begin{theorem}\label{th3}
{\it\, The conclusion of Theorem~\ref{th3_1} remains true, if, under
the conditions of this theorem, we replace~5) by one of the
following conditions:

\medskip
5.1) the function $Q$ has a finite mean oscillation at every point
$b\in \partial D;$

\medskip
5.2) $q_{b}(r)\,=\,O\left(\left[\log{\frac1r}\right]^{n-1}\right)$
as $r\rightarrow 0$ for every point $b\in \partial D;$

\medskip 5.3) the condition
\begin{equation}\label{eq6}
\int\limits_{0}^{\delta(b)}\frac{dt}{t^{\frac{n-1}{\alpha-1}}
q_{b}^{\,\prime\,\frac{1}{\alpha-1}}(t)}=\infty
\end{equation}
holds for every point $b\in \partial D$ and some $\delta(b)>0.$ Then
$f$ has a continuous extension
$\overline{f}:\overline{D}_P\rightarrow\overline{D^{\,\prime}},$
moreover, $\overline{f}(\overline{D}_P)=\overline{D^{\,\prime}}.$ }
\end{theorem}

\medskip
For Sobolev classes on the plane, Theorems~\ref{th3_1}--\ref{th3}
have a somewhat simpler form.

\medskip
\begin{theorem}\label{th1_1}
{\it\, Let $1<\alpha\leqslant 2,$ let $D$ and $D^{\,\prime}$ be
bounded domains in ${\Bbb R}^2,$ let $Q:D\rightarrow [0, \infty]$ be
a Lebesgue measurable function. Assume that, $D$ be a regular
domain. Let $f:D\rightarrow D^{\,\prime}$ be a bounded open discrete
mapping of the class $W^{1, 1}_{\rm loc}(D),$ $f(D)=D^{\,\prime}.$
In addition, assume that $C(f,
\partial D)\subset E_*$ for some closed (in the topology of ${\Bbb R}^n$) set $E_*\subset
\overline{D^{\,\prime}}$ and $f^{\,-1}(E_*)=E$ for some closed (in
the topology of $D$) subset $E\subset \overline{D}.$ In addition,
assume that the conditions 1)--3) of Theorem~\ref{th3_1} holds.
Assume also that $K_{I, \alpha}(x, f)\leqslant Q(x)$ a.e. and the
assumption~5) in Theorem~\ref{th3_1} is satisfied for $n=2.$ Then
$f$ has a continuous extension
$\overline{f}:\overline{D}_P\rightarrow\overline{D^{\,\prime}},$
moreover, $\overline{f}(\overline{D}_P)=\overline{D^{\,\prime}}.$ }
\end{theorem}

\medskip
\begin{theorem}\label{th1}
{\it\, The conclusion of Theorem~\ref{th1_1} remains true, if, under
the conditions of this theorem, we replace~5) by one of the
following conditions~5.1)--5.3) listed above in Theorem~\ref{th3}.}
\end{theorem}

\medskip
\begin{remark}\label{rem1}
Put
$$\widetilde{Q}(x)=\begin{cases}Q(x), & Q(x)\geqslant 1\\ 1, & Q(x)<1\end{cases}\,.$$
If $Q$ satisfies~(\ref{eq1D}) for some $\Phi\colon\overline{{\Bbb
R}^{+}}\rightarrow\overline{{\Bbb R}^{+}},$ then also
$\widetilde{Q}(x)$ satisfies~(\ref{eq1D}). Indeed,
$$
\int\limits_D\Phi(\widetilde{Q}(x))\frac{dm(x)}{\left(1+|x|^2\right)^n}=
\int\limits_{\{x\in D: Q(x)< 1
\}}\Phi(\widetilde{Q}(x))\frac{dm(x)}{\left(1+|x|^2\right)^n}+$$$$+
\int\limits_{\{x\in D: Q(x)\geqslant 1\}
}\Phi(\widetilde{Q}(x))\frac{dm(x)}{\left(1+|x|^2\right)^n}\leqslant
M_0+\Phi(1)\int\limits_{D}\frac{dm(x)}{\left(1+|x|^2\right)^n}=M^{\,\prime}_0<\infty\,.$$
\end{remark}

\section{Preliminaries}

Let $x_0\in {\Bbb R}^n,$ $0<r_1<r_2<\infty,$
\begin{equation}\label{eq1ED}
S(x_0,r) = \{ x\,\in\,{\Bbb R}^n : |x-x_0|=r\}\,, \quad B(x_0, r)=\{
x\,\in\,{\Bbb R}^n : |x-x_0|<r\}\end{equation}
and
\begin{equation}\label{eq1**}
A=A(x_0, r_1,r_2)=\left\{ x\,\in\,{\Bbb R}^n:
r_1<|x-x_0|<r_2\right\}\,.\end{equation}
The most important element of the paper is the connection between
Sobolev and Orlicz-Sobolev classes and lower and ring $Q$-mappings.
The theoretical part of this connection is mainly established
in~\cite{Sev$_2$}. Our immediate goal is to present below the
corresponding definitions, as well as facts revealing the indicated
connection.

\medskip
If $\rho:{\Bbb R}^n\rightarrow[0,\infty]$ is a Borel function, then
its {\it integral over a $k$-dimensional surface} $S$ in ${\Bbb
R}^n$, $n\geqslant 2,$ is defined by the equality
\begin{equation}\label{eq8.2.5} \int\limits_S \rho\ d{\cal {A}}\
:=\ \int\limits_{{\Bbb R}^n}\rho(y)\:N(S, y)\ d{\cal H}^ky\
\,,\end{equation}
where $N(S, y)$ denotes the multiplicity function of $S.$ Given a
family ${\cal S}$ of such $k$-dimensional surfaces $S$ in ${\Bbb
R}^n$, a Borel function $\ \rho:{\Bbb R}^n\rightarrow[0,\infty]$ is
called {\it admissible} for $\cal S$, abbr.
$\rho\in\mathrm{adm}\,\cal S$, if
\begin{equation}\label{eq8.2.6}
\int\limits_S\rho^k\ d{\cal{A}}\ \geqslant\ 1\end{equation} for
every $S\in\cal S$. Given $p\in[k,\infty)$, the {\it $p$-modulus} of
$\cal S$ is the quantity
\begin{equation}\label{M} M_p({\cal S})\ =\
\inf_{\rho\in\mathrm{adm}\,\cal S}\int\limits_{{\Bbb R}^n}\rho^p(x)\
dm(x)\,.\end{equation}
The following definition is from~\cite[Chapter~9]{MRSY}. Let $D$ and
$D^{\,\prime}$ be domains in ${\Bbb R}^n$ with $n\geqslant 2$.
Suppose that $x_0\in\overline {D}\setminus\{\infty\}$ and $Q\colon
D\rightarrow(0,\infty)$ is a Lebesgue measurable function. A mapping
$f:D\rightarrow D^{\,\prime}$ is called a {\it lower $Q$-mapping at
a point $x_0$ relative to $p$-modulus} if
\begin{equation}\label{eq1A}
M_p(f(\Sigma_{\varepsilon}))\geqslant \inf\limits_{\rho\in{\rm
ext}_p\,{\rm adm}\Sigma_{\varepsilon}}\int\limits_{D\cap A(x_0,
\varepsilon, r_0)}\frac{\rho^p(x)}{Q(x)}\,dm(x)
\end{equation}
for $A(x_0, \varepsilon, r_0)=\{x\in {\Bbb R}^n\,:\,
\varepsilon<|x-x_0|<r_0\}$, $r_0\in(0,d_0)$, $d_0=\sup\limits_{x\in
D}|x-x_0|$, where $\Sigma_{\varepsilon}$ is the family of all
intersections of the spheres $S(x_0, r)$ with the domain $D$, $r\in
(\varepsilon, r_0)$. If $p=n$, we say that $f$ is a lower
$Q$-mapping at $x_0$. We say that $f$ is a lower $Q$-mapping
relative to $p$-modulus in $B\subset \overline {D}$ if (\ref{eq1A})
is true for all $x_0\in B$. The following statement can be proved
much as Theorem~9.2 in \cite{MRSY}, so we omit the arguments.

\medskip
\begin{lemma}\label{lem4} {\it\, Let $D$,
$D^{\,\prime}\subset\overline{{\Bbb R}^n}$, let $x_0\in\overline
{D}\setminus\{\infty\}$, and let $Q$ be a Lebesgue measurable
function. A mapping $f\colon D\rightarrow D^{\,\prime}$ is a lower
$Q$-mapping relative to the $p$-modulus at a point $x_0$, $p>n-1$,
if and only if $M_p(f(\Sigma_{\varepsilon}))\geqslant
\int\limits_{\varepsilon}^{r_0} \frac{dr}{\|\,Q\|_{s}(r)}$ for all
$\varepsilon\in(0,r_0),\ r_0\in(0,d_0)$, $d_0=\sup\limits_{x\in
D}|x-x_0|$, $s=\frac{n-1}{p-n+1}$, where, as above,
$\Sigma_{\varepsilon}$ denotes the family of all intersections of
the spheres $S(x_0, r)$ with $D$, $r\in (\varepsilon, r_0)$, $\|
Q\|_{s}(r)=\left(\int\limits_{D(x_0,r)}Q^{s}(x)\,d{\mathcal{A}}\right)^{\frac{1}{s}}$
is the $L_{s}$-norm of $Q$ over the sphere ${D(x_0,r)=\{x\in D\,:\,
|x-x_0|=r\}=D\cap S(x_0,r)}\,.$}
\end{lemma}

\medskip
Let
\begin{equation}\label{eq1_A_2}
N(y, f, A)\,=\,{\rm card}\,\left\{x\in A: f(x)=y\right\}\,,
\end{equation}
\begin{equation}\label{eq1_A_3}
N(f, A)\,=\,\sup\limits_{y\in{\Bbb R}^n}\,N(y, f, A)\,.
\end{equation}
The following lemma holds, see e.g. \cite[Lemma~5.1]{Sev$_4$}.
\begin{lemma}\label{thOS4.1} {\it\,Let $D$ be a domain in ${\Bbb R}^n$,
$n\geqslant  3$, and let $\varphi\colon (0,\infty)\rightarrow
(0,\infty)$ be a monotone nondecreasing function satisfying
(\ref{eq1_A_10}). If $p>n-1$, then every open discrete mapping
$f\colon D\rightarrow {\Bbb R}^n$ of class $W^{1,\varphi}_{{\rm
loc}}$ with finite distortion and such that $N(f, D)<\infty$ is a
lower $Q$-mapping relative to the $p$-modulus at every point
$x_0\in\overline {D},$ where
$$
Q(x)=N(f, D)\cdot K^{\frac{p-n+1}{n-1}}_{I, \alpha}(x, f),
$$
$\alpha:=\frac{p}{p-n+1}$, the inner dilation $K_{I,\alpha}(x, f)$
for~$f$ at~$x$ of order~$\alpha$ is defined by \eqref{eq0.1.1A}, and
the multiplicity $N(f, D)$ is defined by the
relation~(\ref{eq1_A_3}).}
\end{lemma}

\medskip
The statement similar to Lemma~\ref{thOS4.1} holds for $n=2,$ but
for Sobolev classes (see, e.g., \cite[Theorem~4]{Sev$_1$}).

\medskip
\begin{lemma}\label{lem3} {\it\,Let $D$ be a domain in ${\Bbb R}^2$,
$n\geqslant  2$, and let $p>1.$ Then every open discrete mapping
$f\colon D\rightarrow {\Bbb R}^2$ of the class $W^{1,1}_{{\rm loc}}$
with finite distortion and such that $N(f, D)<\infty$ is a lower
$Q$-mapping relative to the $p$-modulus at every point
$x_0\in\overline {D}$ for
$$
Q(x)=N(f, D)\cdot K^{p-1}_{I, \alpha}(x, f),
$$
$\alpha:=\frac{p}{p-1}$, where the inner dilation $K_{I,\alpha}(x,
f)$ for~$f$ at~$x$ of order~$\alpha$ is defined by \eqref{eq0.1.1A},
and the multiplicity $N(f, D)$ is defined by the
relation~(\ref{eq1_A_3}).}
\end{lemma}

\medskip
The following important information concerning the capacity of a
pair of sets relative to a domain can be found in Ziemer's paper
\cite{Zi$_1$}. Let $G$ be a bounded domain in ${\Bbb R}^n$, and let
$C_{0}, C_{1}$ be nonintersecting compact subsets of the closure of
$G$. Put $R=G \setminus (C_{0} \cup C_{1})$ and $R^{\,*}=R \cup
C_{0}\cup C_{1}$, then the {\it $p$-capacity of the pair $C_{0},
C_{1}$ relative to the closure of~$G$} is defined to be the quantity
$C_p[G, C_{0}, C_{1}] = \inf \int\limits_{R} | \nabla u|^{p}\,
dm(x)$, where the infimum is taken over all functions $u$ continuous
on $R^{\,*}$, $u\in ACL(R)$, with $u=1$ on $C_{1}$ and $u=0$ on
$C_{0}$. Such functions will be called admissible for $C_p [G,
C_{0}, C_{1}]$. A set $\sigma \subset {\Bbb R}^n$ is said to {\it
separate $C_{0}$ and $C_{1}$ in $R^{\,*}$} if $\sigma \cap R$ is
closed in $R$ and there exist nonintersecting set $A$ and $B$ open
in $R^{\,*} \setminus \sigma$ and such that $R^{\,*} \setminus
\sigma = A\cup B$, $C_{0}\subset A$, and $C_{1} \subset B$. Let
$\Sigma$ denote the class of all sets separating $C_{0}$ and $C_{1}$
in $R^{\,*}$. Putting $p^{\,\prime} = p/(p-1)$, we introduce the
quantity
\begin{equation}\label{eq4.23}
 \widetilde{M}_{p^{\,\prime}}(\Sigma)=\inf_{\rho\in
\widetilde{\rm adm} \Sigma} \int\limits_{{\Bbb
R}^n}\rho^{\,p^{\,\prime}}\,dm(x)\,,
\end{equation}
where the formula $\rho\in \widetilde{\rm adm}\Sigma$ means that
$\rho$ is a nonnegative Borel function on ${\Bbb R}^n$ such that
\begin{equation} \label{eq13.4.13}
\int\limits_{\sigma \cap R}\rho \,d{\mathcal H}^{n-1} \geqslant
1\quad\forall\, \sigma \in \Sigma.
\end{equation}
Observe that, by Ziemer's result,
\begin{equation}\label{eq3_B}
\widetilde{M}_{p^{\,\prime}}(\Sigma)=C_p[G , C_{0} ,
C_{1}]^{\,-1/(p-1)},
\end{equation}
see \cite[Theorem~3.13]{Zi$_1$} for $p=n$ and \cite[p.~50]{Zi$_2$}
for $1<p<\infty$. We also observe that, by a result of Hesse,
\begin{equation}\label{eq4_A}
M_p(\Gamma(E, F, D))= C_p[D, E, F]
\end{equation}
under the condition $(E\cup F)\cap\partial D=\varnothing$
see~\cite[Theorem~5.5]{Hes}. Shlyk has proved that the condition
$(E\cup F)\cap\partial D=\varnothing$ may be removed,
see~\cite[Theorem~1]{Shl}.

\medskip
The following statement holds, see e.g. \cite[Lemma~7.4,
Ch.~7]{MRSY} for $p=n$ and \cite[Lemma~2.2]{Sal} for $p\ne n.$

\medskip
\begin{proposition}\label{pr1A}
{\,\it Let $x_0 \in {\Bbb R}^n,$ $Q(x)$ be a Lebesgue measurable
function, $Q:{\Bbb R}^n\rightarrow [0, \infty],$ $Q\in
L_{loc}^1({\Bbb R}^n).$ We set $A:=A(x_0, r_1, r_2)=\{ x\,\in\,{\Bbb
R}^n : r_1<|x-x_0|<r_2\}$ and
$\eta_0(r)=\frac{1}{Ir^{\frac{n-1}{p-1}}q_{x_0}^{\frac{1}{p-1}}(r)},$
where $I:=I=I(x_0,r_1,r_2)=\int\limits_{r_1}^{r_2}\
\frac{dr}{r^{\frac{n-1}{p-1}}q_{x_0}^{\frac{1}{p-1}}(r)}$ and
$q_{x_0}(r):=\frac{1}{\omega_{n-1}r^{n-1}}\int\limits_{|x-x_0|=r}Q(x)\,d{\mathcal
H}^{n-1}$ is the integral average of the function $Q$ over the
sphere $S(x_0, r).$ Then
\begin{equation*}\label{eq10A_1}
\frac{\omega_{n-1}}{I^{p-1}}=\int\limits_{A} Q(x)\cdot
\eta_0^p(|x-x_0|)\ dm(x)\leqslant\int\limits_{A} Q(x)\cdot
\eta^p(|x-x_0|)\ dm(x)
\end{equation*}
for any Lebesgue measurable function $\eta :(r_1,r_2)\rightarrow
[0,\infty]$ such that
$\int\limits_{r_1}^{r_2}\eta(r)\,dr=1. $ Moreover, (\ref{eq10A_1})
holds for similar functions $\eta$ with
$\int\limits_{r_1}^{r_2}\eta(r)\,dr\geqslant 1$ (see, e.g.,
\cite[Remark~3.1]{Sev$_3$}).}
\end{proposition}

The following lemma was proved earlier for balls,
see~\cite[Theorem~5]{Sev$_2$}. Apparently, its presentation in
application to ``domains with prime ends'' is made for the first
time.

\medskip
\begin{lemma}\label{lem4A}
{\,\it Let $f:D\rightarrow {\Bbb R}^n$ be a bounded, open and
discrete mapping and let $G\subset D$ be domain in which $f$ is a
closed lower $Q$-mapping with respect to $p$-modulus, $Q\in
L_{loc}^{\frac{n-1}{p-n+1}}({\Bbb R}^n),$ $n-1<p,$ and
$\alpha:=\frac{p}{p-n+1}.$ Let $D$ be a regular domain, let $P_0\in
E_D$ and let $d_m,$ $m=1,2,\ldots,$ be a sequence of nested domains
corresponding to $P_0$ such that the cuts $\sigma_m$ corresponding
to $d_m,$ $m=1,2,\ldots, $ lies on the spheres $S(x_0, r_m),$ where
$x_0\in \partial D$ and $r_m\rightarrow 0$ as $m\rightarrow\infty.$
Then, given $\varepsilon_1>0$ and $m\in {\Bbb N}$ there is
$k_0=k_0(m, \varepsilon_1)\in {\Bbb N}$ such that the relation
\begin{equation}\label{eq3A_2}
M_{\alpha}(f(\Gamma(C_1, C_2, G)))\leqslant \int\limits_{A(x_0, r_k,
\varepsilon_1)}Q^{\frac{n-1}{p-n+1}}(x)
\eta^{\alpha}(|x-x_0|)\,dm(x)
\end{equation}
holds for any $k\geqslant k_0(m, \varepsilon_1),$ every compactum
$C_2\subset G\setminus d_m$ with $d((f|_{G})^{\,-1}(f(C_2)),
x_0)\geqslant \varepsilon_1,$ $(f|_{G})^{\,-1}(f(C_2))=\{x\in G:
\,\exists\,y\in f(C_2): f(x)=y\},$ and each compactum $C_1\subset
d_k\cap G,$ where $A(x_0, r_k, \varepsilon_1)=\{x\in {\Bbb R}^n:
r_k<|x-x_0|<\varepsilon_1\}$ and $\eta: (r_k,
\varepsilon_1)\rightarrow [0,\infty]$ is an arbitrary Lebesgue
measurable function such that
\begin{equation}\label{eq6B}
\int\limits_{r_k}^{\varepsilon_1}\eta(r)\,dr=1\,.
\end{equation}
}
\end{lemma}
\begin{proof}
In general, we focus on the approach used in the proof
of~\cite[Theorem~5]{Sev$_2$}, cf.~\cite[Theorem~5.5]{Sev$_4$}.
Firstly we prove that the sets $f(C_2)$ and $\overline{f(B(x_0,
r)\cap G)}$ do not intersect for every $r\in (0, \varepsilon_1)$.
Suppose the contrary, namely, there exists $\zeta_0\in f(C_2)\cap
\overline{f(B(x_0, r)\cap G)}$. Then
$\zeta_0=\lim\limits_{l\rightarrow\infty} \zeta_l$, where
$\zeta_l\in f(B(x_0, r)\cap G)$. We have $\zeta_l=f(\xi_l)$,
$\xi_l\in B(x_0, r)\cap G$. Since $D$ is bounded by the condition,
$\overline{G}$ is compact. Consequently, the sequence $\xi_l$
contains a convergent subsequence $\xi_{l_m}\rightarrow
\xi_0\in\overline{B(x_0, r)\cap G}$ as $m\rightarrow\infty.$ The
case of $\xi_0\in
\partial G$ is impossible because $f$ is closed in $G$,
consequently, preserves the boundary: $C(f,\partial G)\subset
\partial f(G)$, see e.g. \cite[Theorem~3.3]{Vu}.
But $\zeta_0\in f(C_2)$ and therefore $\zeta_0$ is an inner point of
$f(G)$. Let $\xi_0\in G.$ Then $f(\xi_0)=\zeta_0$ because $f$ is
continuous. But then $\xi_0\in B(x_0, \varepsilon_1)\cap G$ and at
the same time $\xi_0\in (f_G)^{\,-1}(f(C_2))$, which contradicts the
choice of $\varepsilon_1$. Thus, $f(C_2)\cap \overline{f(B(x_0,
r)\cap G)}=\varnothing$ and, therefore,
\begin{equation}\label{eq8_A}
f(C_2)\subset f(G)\setminus \overline{f(B(x_0, r)\cap G)},\qquad
r\in (0, \varepsilon_1).
\end{equation}
Let $\sigma_k$ is a cut of $D$ in $P_0$ which corresponds to $d_k$
and lying on the sphere $S(x_0, r_k),$ $r_k>0.$ Since
$r_k\rightarrow 0$ as $k\rightarrow \infty$ by the conditions of the
lemma, (\ref{eq8_A}) implies that $f(C_2)$ and $f(\sigma_k\cap G)$
are disjoint for sufficiently large $k.$ Let
$k_1=k_1(\varepsilon_1)$ be a number such that $r_{k}<\varepsilon_1$
for any $k\geqslant k_1(\varepsilon_1).$ Now we set $k_0=k_0(m,
\varepsilon_1)=\max\{m+1, k_1(\varepsilon_1)\}.$ Let $k\geqslant
k_0.$

Observe that, $\sigma_k\cap G\ne\varnothing.$ Indeed, we join points
$x\in C_2\subset G\setminus d_m\subset G\setminus d_k$ and $y\in
C_1\subset d_k\cap G$ by a path $\gamma$ in $G.$ Now, $|\gamma|\cap
(\partial d_k\cap G)\ne\varnothing$ (see e.g. \cite[Theorem~1.I.5,
\S46]{Ku}). Now, there is $z\in |\gamma|\cap (\partial d_k\cap G).$
However, in this case, $z\in (\partial d_k\cap G)\subset
\sigma_k\cap G,$ as required. We also observe that, for every $r\in
(r_k, \varepsilon_1)$ the set
$$
A_r:=\partial (f(B(x_0, r)\cap G))\cap f(G)
$$
separates $f(C_2)$ and $f(\sigma_k\cap G)$ in $f(G)$. Indeed,
$$
f(G)=B_r\cup A_r\cup C_r\quad\forall\quad r\in (r_k, \varepsilon_1),
$$
where the sets $B_r:=f(B(x_0, r)\cap G)$ and $C_r:=f(G)\setminus
\overline{f(B(x_0, r)\cap G)}$ are open in $f(G)$, $f(\sigma_k\cap
G)\subset B_r$, $f(C_2)\subset C_r$, and $A_r$ is closed in $f(G)$.

\smallskip
Let $\Sigma_k$ be the family of all sets separating $f(\sigma_k\cap
G)$ and $f(C_2)$ in $f(G)$. Since $f$ is a closed mapping in $G,$ we
observe that
\begin{equation}\label{eq7A} (\partial f(B(x_0, r)\cap
D))\cap f(G)\subset f(S(x_0, r)\cap G),\quad r>0.
\end{equation}
Indeed, let $\zeta_0\in(\partial f(B(x_0,r)\cap G))\cap f(G)$. Then
there exists a sequence $\zeta_l\in f(B(x_0, r)\cap G)$ such that
$\zeta_l\rightarrow \zeta_0$ as $l\rightarrow \infty$, where
$\zeta_l=f(\xi_l)$, $\xi_l\in B(x_0,r)\cap G$. Without loss of
generality, we may assume that $\xi_l\rightarrow \xi_0$ as
$l\rightarrow\infty$. Note that, the case when $\xi_0\in
\partial G$ is impossible, because, in this case, $\zeta_0\in C(f,\partial
G)\subset \partial f(G)$, which contradicts the choice of $\zeta_0.$
Thus, $\xi_0\in G$. Two situations may occur: 1) $\xi_0\in B(x_0 ,
r)\cap G$ and 2) $\xi_0\in S(x_0 , r)\cap G$. Case~1) is impossible
because then $f(\xi_0)=\zeta_0$ and $\zeta_0$ is an inner point of
$f(B(x_0, r)\cap G)$, which contradicts the choice of $\zeta_0$.
Thus,~(\ref{eq7A}) is proved.

\smallskip
Here and below, the unions of the form
$$
\bigcup_{r\in (r_1, r_2)}\partial f(B(x_0, r)\cap G)\cap f(G)
$$
are understood as families of sets. Let
$$
\rho^{n-1}\in \widetilde{{\rm adm}}\bigcup_{r\in (r_k,
\varepsilon_1)}
\partial f(B(x_0, r)\cap G)\cap f(G)
$$
in the sense of (\ref{eq13.4.13}), then also
$$
\rho\in {\rm adm}\bigcup_{r\in (r_k, \varepsilon_1)}
\partial f(B(x_0, r)\cap G)\cap f(G)
$$
in the sense of~(\ref{eq8.2.6}) with $k=n-1$. By (\ref{eq7A}), we
see that
$$\rho\in {\rm adm}\bigcup\limits_{r\in (r_k,
\varepsilon_1)} f(S(x_0, r)\cap G)\,.$$
Since also $\widetilde{M}_q(\Sigma_k)\geqslant M_{q(n-1)}(\Sigma_k)$
for every $q\geqslant  1$, we have
\begin{equation}\label{eq5A}
\begin{split}
\widetilde{M}_{p/(n-1)}(\Sigma_k) &\geqslant
\widetilde{M}_{p/(n-1)}\bigg(\bigcup_{r\in (r_k, \varepsilon_1)}
   \partial f(B(x_0, r)\cap G)\cap f(G)\bigg)
\\
&\geqslant \widetilde{M}_{p/(n-1)} \bigg(\bigcup_{r\in (r_k,
\varepsilon_1)} f(S(x_0, r)\cap G)\bigg)
\\
&\geqslant  M_{p}\bigg(\bigcup_{r\in (r_k, \varepsilon_1)} f(S(x_0,
r)\cap G)\bigg).
\end{split}
\end{equation}

By (\ref{eq3_B}) and (\ref{eq4_A}), it follows that for $p>n-1$ we
have
\begin{equation}\label{eq6A}
\widetilde{M}_{p/(n-1)}(\Sigma_k)=\frac{1}{(M_{\alpha}(\Gamma(f(\sigma_k\cap
G), f(C_2), f(G))))^{1/(\alpha-1)}}.
\end{equation}
Now, by Lemma~\ref{lem4} we obtain that
\begin{gather}\nonumber M_{p}\left(\bigcup\limits_{r\in (r_k, \varepsilon_1)} f(S(x_0,
r)\cap G)\right)\\
\label{eq8B_1} \geqslant \int\limits_{r_k}^{\varepsilon_1}
\frac{dr}{\Vert\,Q\Vert_{s}(r)}= \int\limits_{r_k}^{\varepsilon_1}
\frac{dt}{\omega^{\frac{p-n+1}{n-1}}_{n-1}
t^{\frac{n-1}{\alpha-1}}\widetilde{q}_{x_0}^{\,\frac{1}{\alpha-1}}(t)}\quad\forall\,\,
i\in {\Bbb N}\,, \qquad s=\frac{n-1}{p-n+1}\,,\end{gather} where
$\alpha=\frac{p}{p-n+1},$
$\Vert
Q\Vert_{s}(r)=\left(\int\limits_{D(x_0,r)}Q^{s}(x)\,d{\mathcal{A}}\right)^{\frac{1}{s}}$
is $L_{s}$-norm of the function $Q$ over the sphere $S(x_0,r)\cap G$
and
$\widetilde{q}_{x_0}(r):=\frac{1}{\omega_{n-1}r^{n-1}}\int\limits_{|x-x_0|=r}Q^s(x)\,d{\mathcal
H}^{n-1}.$ Then from (\ref{eq5A})--(\ref{eq8B_1}) it follows that
\begin{equation}\label{eq9C}
M_{\alpha}(\Gamma(f(\sigma_k\cap G), f(C_2), f(G)) \leqslant
\frac{\omega_{n-1}}{I^{\alpha-1}}\,,
\end{equation}
where $I=\int\limits_{r_k}^{\varepsilon_1}\
\frac{dr}{r^{\frac{n-1}{\alpha-1}}\widetilde{q}_{x_0}^{\frac{1}{\alpha-1}}(r)}.$
Since $\sigma_k$ is a cut corresponding to a domain $d_k$ and
$C_2\subset G\setminus d_k\subset D\setminus d_k$ by the
construction, we obtain that $\Gamma(C_1, C_2,
G)>\Gamma(\sigma_k\cap G, C_2, G)$ and, consequently, $f(\Gamma(C_1,
C_2, G))>f(\Gamma(\sigma_k\cap G, C_2, G))\subset
\Gamma(f(\sigma_k\cap G)), f(C_2), f(G)).$ Thus,
\begin{equation}\label{eq1D_1}
M_{\alpha}(f(\Gamma(C_1, C_2, G))\leqslant
M_{\alpha}(\Gamma(f(\sigma_k), f(C_2), f(G)))\,.
\end{equation}
Combining~(\ref{eq9C}) and (\ref{eq1D_1}), we obtain that
\begin{equation}\label{eq1E}
M_{\alpha}(f(\Gamma(C_1, C_2, G)) \leqslant
\frac{\omega_{n-1}}{I^{\alpha-1}}\,.
\end{equation}
The proof is completed by applying Proposition~\ref{pr1A}.~$\Box$
\end{proof}

\medskip
Let $D\subset {\Bbb R}^n,$ $f:D\rightarrow {\Bbb R}^n$ be a discrete
open mapping, $\beta: [a,\,b)\rightarrow {\Bbb R}^n$ be a path, and
$x\in\,f^{\,-1}(\beta(a)).$ A path $\alpha: [a,\,c)\rightarrow D$ is
called a {\it maximal $f$-lifting} of $\beta$ starting at $x,$ if
$(1)\quad \alpha(a)=x\,;$ $(2)\quad f\circ\alpha=\beta|_{[a,\,c)};$
$(3)$\quad for $c<c^{\prime}\leqslant b,$ there is no a path
$\alpha^{\prime}: [a,\,c^{\prime})\rightarrow D$ such that
$\alpha=\alpha^{\prime}|_{[a,\,c)}$ and $f\circ
\alpha^{\,\prime}=\beta|_{[a,\,c^{\prime})}.$ Here and in the
following we say that a path $\beta:[a, b)\rightarrow
\overline{{\Bbb R}^n}$ converges to the set $C\subset
\overline{{\Bbb R}^n}$ as $t\rightarrow b,$ if $h(\beta(t),
C)=\sup\limits_{x\in C}h(\beta(t), C)\rightarrow 0$ as $t\rightarrow
b.$ The following is true (see~\cite[Lemma~3.12]{MRV}).

\medskip
\begin{proposition}\label{pr3}
{\it\, Let $f:D\rightarrow {\Bbb R}^n,$ $n\geqslant 2,$ be an open
discrete mapping, let $x_0\in D,$ and let $\beta: [a,\,b)\rightarrow
{\Bbb R}^n$ be a path such that $\beta(a)=f(x_0)$ and such that
either $\lim\limits_{t\rightarrow b}\beta(t)$ exists, or
$\beta(t)\rightarrow \partial f(D)$ as $t\rightarrow b.$ Then
$\beta$ has a maximal $f$-lifting $\alpha: [a,\,c)\rightarrow D$
starting at $x_0.$ If $\alpha(t)\rightarrow x_1\in D$ as
$t\rightarrow c,$ then $c=b$ and $f(x_1)=\lim\limits_{t\rightarrow
b}\beta(t).$ Otherwise $\alpha(t)\rightarrow \partial D$ as
$t\rightarrow c.$}
\end{proposition}

\section{Proof of Theorem~\ref{th3}}

It is sufficiently to prove~Theorem~\ref{th3} in the case~5.3),
i.e., when the relation~(\ref{eq6}) holds. Indeed, by arguments
given between formulae~(5.17)--(5.18) in \cite{Sev$_4$}, the
condition~5.1) implies~5.3). Besides that, 5.2) implies~5.3), as
well.

\medskip
By the assumption~2) of the theorem, there is a sequence of
neighborhoods $V_k$ of the prime end $P_0,$ $V_k\subset B(P_0,
2^{\,-k}),$ $k=1,2,\ldots ,$ such that $V_k\cap D$ is connected and
$(V_k\cap D)\setminus E$ consists at most of $m$ components,
$1\leqslant m<\infty.$ Since $D$ is a regular domain, there exists a
quasiconformal mapping $\varphi$ of $D$ onto some domain $D_0$ with
a locally quasiconformal boundary. The latter means that there is a
sequence $l_k>0,$ $k=1,2,\ldots ,$ such that $l_k\rightarrow 0$ as
$k\rightarrow\infty,$ and a point $z_0\in
\partial D_0$ such that $\varphi(V_k)\subset B(z_0, l_k),$
$k=1,2,\ldots ,$ $\varphi(V_k)\cap D_0$ is connected and $(V_k\cap
D)\setminus \varphi(E)$ consists at most of $m$ components,
$1\leqslant m<\infty.$

We show that, $D_0\setminus \varphi(E)$ consists of finite number of
components $D^{\,*}_1,\ldots, D^{\,*}_s,$ $1\leqslant s<\infty.$
Assume the contrary, that $D_0\setminus \varphi(E)$ consists of
infinite number of components $D^{\,*}_1, D^{\,*}_2,\ldots .$ Let
$z_i\in D^{\,*}_i,$ $i=1,2,\ldots .$ Since $D_0$ is bounded, there
is a subsequence $z_{i_k},$ $k_1,2,\ldots ,$ converging to some
point $z^{\,*}_0\in \overline{D_0}.$ On the other hand, there is a
neighborhood $V^{\,*}$ of the point $z^{\,*}_0$ such that
$(V^{\,*}\cap D)\setminus \varphi(E)$ consists of $m$ components.
Therefore, there is at least one such a component $K$ intersecting
infinitely many components $D^{\,*}_{i_1}, D^{\,*}_{i_2},\ldots .$
However, by the assumptions of the theorem there is a neighborhood
$V^{\,*}$ of $z^{\,*}_0$ such that $(V^{\,*}\cap D_0)\setminus
\varphi(E)$ consists of finite number of components
$V^{\,*}_1,\ldots , V^{\,*}_m.$ Now, there is $m_0\in[1, m]$ such
that $V^{\,*}_{m_0}$ intersects infinitely many $D^{\,*}_i,$
contradiction. Thus, $D^{\,*}_1,\ldots, D^{\,*}_s,$ $1\leqslant
s<\infty,$ are all the components of $D_0\setminus \varphi(E),$ as
required. It follows from that, $D\setminus E$ consists of some $s$
components, $D_1,\ldots, D_s,$ $1\leqslant s<\infty,$ as well.

\medskip
Now, let us to prove that $f$ is closed in each $D_i,$
$i=1,2,\ldots, $ i.e. $f(S)$ is closed in $f(D_i)$ whenever $S$ is
closed in $D_i.$ Since $f$ is open and discrete, it is sufficient to
prove that $f$ is boundary preserving, that is $C(f,
\partial D_i)\subset \partial f(D_i)$ (see, e.g.,
\cite[Theorem~3.3]{Vu}). Indeed, let $z_k\in D_i,$ $k=1,2,\ldots ,$
let $z_k\rightarrow z_0\in\partial D_i$ and let $f(z_k)\rightarrow
w_0$ as $k\rightarrow\infty.$ We need to prove that $w_0\in
\partial f(D_i).$ Assume the contrary, i.e. $w_0\in f(D_i).$
There are two cases: $z_0\in \partial D$ or $z_0\in D.$ 1) In the
first case, when $z_0\in \partial D,$ we obtain that $w_0\in C(f,
\partial D)\subset E_*.$ But since $w_0\in f(D_i),$ there is
$\zeta_i\in D_i$ such that $f(\zeta_i)=w_0.$ Now, since $w_0\in
E_*,$ we obtain that $\zeta_i\in f^{\,-1}(E_*)=E$ and, consequently,
$\zeta_i\not\in D_i,$ contradiction. 2) Let us consider the second
case, when $z_0\in D.$ Now, by the definition of $D_i,$ we have that
$z_0\in E.$ Now, $w_0\in f(E)\subset E_*.$ But since $w_0\in
f(D_i),$ there is $\zeta_i\in D_i$ such that $f(\zeta_i)=w_0.$ Since
$w_0\in E_*,$ we have that $\zeta_i\in f^{\,-1}(E_*)=E$ and,
consequently, $\zeta_i\not\in D_i,$ contradiction.

\medskip
Let $K$ be a component of $D^{\,\prime}\setminus E_*$ consisting
$f(D_i).$ Observe that $f(D_i)=K.$ Indeed, $f(D_i)\subset K$ by the
definition, because $f^{\,-1}(E_*)=E$ and $D_i\subset D\setminus E.$
Let us to prove that $K\subset f(D_i).$ Let us prove this inclusion
by the contradiction, i.e., let $a_0\in K\setminus f(D_i).$ Chose
$b_0\in f(D_i)$ and join the points $b_0$ and $a_0$ by a path
$\beta:[0, 1]\rightarrow K.$ Let $\alpha:[0, c)\rightarrow D$ be a
maximal $f$-lifting of $\beta$ starting at $c_0:=f^{\,-1}(b_0)\cap
D_i$ (this lifting exists by Proposition~\ref{pr3}). By the same
proposition either one of two cases are possible:
$\alpha(t)\rightarrow x_1$ as $t\rightarrow c,$ or
$\alpha(t)\rightarrow \partial D_i$ as $t\rightarrow c.$ In the
first case, by Proposition~\ref{pr3} we obtain that $c=1$ and
$f(\beta(1))=f(x_1)=a_0$ which contradict with the choice of $a_0.$
In the second case, when $\alpha(t)\rightarrow \partial D_i$ as
$t\rightarrow c,$ we obtain that $f(\beta(c))\in C(f, \partial
D_i)\subset E_*.$ The latter contradicts the definition of $\beta$
because $\beta$ does not contain itself points in $E_*.$

\medskip
Observe that, $D^{\,\prime}\setminus E_*$ consists of finite number
of components. Otherwise, $D^{\,\prime}\setminus
E_*=\bigcup\limits_{i=1}^{\infty}K^{\,\prime}_i,$ $\varnothing \ne
K^{\,\prime}_i\cap K^{\,\prime}_j\ne \varnothing$ for $i\ne j.$  Let
$K_i$ be some component of the set $f^{\,-1}(K^{\,\prime}_i).$ By
the definition $K_i\subset D\setminus E$ and $\varnothing \ne
K_i\cap K_j\ne \varnothing$ for $i\ne j.$ The latter contradicts the
proving above for the number of components of $D\setminus E.$

\medskip
Now, we proceed directly to the proof of the main statement of the
lemma, applying the approach used in the proof of
\cite[Lemma~3.1]{DS}. Suppose the opposite, i.e., the conclusion of
the lemma is not true. Since $D^{\,\prime}$ is bounded, there are
$P_0\in E_D$ and at least two sequences $x_k,$ $y_k\in D,$
$i=1,2,\ldots,$ such that $x_k,$ $y_k\in D,$ $k=1,2,\ldots,$
$x_k\rightarrow P_0,$ $y_k\rightarrow P_0$ as $k\rightarrow \infty,$
and $f(x_k)\rightarrow y,$ $f(y_k)\rightarrow y^{\,\prime}$ as
$k\rightarrow \infty,$ while $y^{\,\prime}\ne y.$
In particular,
\begin{equation}\label{eq1B}
|f(x_k)-f(y_k)|\geqslant \delta>0
\end{equation}
for some $\delta>0$ and all $k\in {\Bbb N}.$
We note that the points $x_k$ and $y_k,$ $k=1,2,\ldots, $ may be
chosen such that $x_k, y_k\not\in E.$ Indeed, since under
condition~1) the set $E$ is nowhere dense in $D,$ there exists a
sequence $x_{ki}\in D\setminus E,$ $i=1,2,\ldots ,$ such that
$x_{ki}\rightarrow x_k$ as $i\rightarrow\infty.$ Put
$\varepsilon>0.$ Due to the continuity of the mapping $f$ at the
point $x_k,$ for the number $k\in {\Bbb N}$ there is a number
$i_k\in {\Bbb N}$ such that $|f(x_{ki_k})-f(x_k)|<\frac{1}{2^k}.$
So, by the triangle inequality
$$|f(x_{ki_k})-y|\leqslant |f(x_{ki_k})-f(x_k)|+
|f(x_k)-y|\leqslant \frac{1}{2^k}+\varepsilon\,,$$
$k\geqslant k_0=k_0(\varepsilon),$ because $f(x_k)\rightarrow y$ as
$k\rightarrow\infty$ and by the choice of $x_k$ and $y.$ Therefore,
$x_k\in D$ may be replaced by $x_{ki_k}\in D\setminus E,$ as
required. We may reason similarly for the sequence $y_k.$

\medskip
Now, by~\cite[Lemma~2.1]{DS} there are subsequences
$\varphi(x_{k_l})$ and $\varphi(y_{k_l}),$ $l=1,2,\ldots ,$
belonging to some sequence of neighborhoods $\varphi(V_l),$
$l=1,2,\ldots ,$ of the point $z_0$ such that ${\rm
diam\,}\varphi(V_l)\rightarrow 0$ as $l\rightarrow\infty$ and, in
addition, any pair $\varphi(x_{k_l})$ and $\varphi(y_{k_l})$ may be
joined by a path $\gamma^{\,*}_l$ in $\varphi(V_l)\cap D_0,$ where
$\gamma^{\,*}_l$ contains at most $m-1$ points in $\varphi(E).$
Without loss of generality, we may assume that the same sequences
$\varphi(x_k)$ and $\varphi(y_k)$ satisfy properties mentioned
above. Let $\gamma^{\,*}_k:[0, 1]\rightarrow D_0,$
$\gamma^{\,*}_k(0)=\varphi(x_k)$ and
$\gamma^{\,*}_k(1)=\varphi(y_k),$ $k=1,2,\ldots .$ Set
$\gamma_k:=\varphi^{\,-1}(\gamma^{\,*}_k).$ We may consider that
$\gamma_k\subset d_k,$ $k=1,2,\ldots,$ where $d_k$ is a sequence of
nested domains corresponding $P_0.$ We may consider that $d_k$
corresponds to cuts $\sigma_k,$ $k=1,2,\ldots ,$ such that
$\sigma_k\subset S(x_0, r_k),$ $x_0\in \partial D$ and
$r_k\rightarrow 0$ as $k\rightarrow \infty,$ see e.g.
\cite[Lemma~3.1]{IS}, cf.~\cite[Lemma~1]{KR}.

\medskip
Observe that, the path $f(\gamma_k)$ contains not more than $m-1$
points in $E_*.$ In the contrary case, there are at least $m$ such
points $b_{1}=f(\gamma_k(t_1)), b_{2}=f(\gamma_k(t_2)),\ldots,
b_m=f(\gamma_k(t_m)),$ $0\leqslant t_1\leqslant t_2\leqslant
\ldots\leqslant t_m\leqslant 1.$ But now the points
$a_{1}=\gamma_k(t_1), a_{2}=\gamma_k(t_2),\ldots, a_m=\gamma_k(t_m)$
are in $f^{\,-1}(E_*)=E$ and simultaneously belong to $\gamma_k.$
This contradicts the definition of $\gamma_k.$

\medskip
Let
\begin{gather*}b_{1}=f(\gamma_k(t_1)),
b_{2}=f(\gamma_k(t_2))\quad,\ldots,\quad b_l=f(\gamma_k(t_l))\,,\\
0\leqslant t_1\leqslant t_2\leqslant \ldots\leqslant t_l\leqslant
1,\qquad 1\leqslant l\leqslant m-1\,,\end{gather*}
be points in $f(\gamma_k)\cap E_*.$ By the relation~(\ref{eq1B}) and
due to the triangle inequality,
\begin{equation}\label{eq6A_1}
\delta\leqslant |f(x_k)-f(y_k)|\leqslant\sum\limits_{r=1}^{l-1}
|f(\gamma_k(t_r))-f(\gamma_k(t_{r+1}))|\,.
\end{equation}
It follows from~(\ref{eq6A_1}) that, there is $1\leqslant
l=l(k)\leqslant m-1$ such that
\begin{equation}\label{eq7}
|f(\gamma_k(t_{l(k)}))-f(\gamma_k(t_{l(k)+1})|\geqslant
\delta/l\geqslant \delta/(m-1)\,.
\end{equation}
Observe that, the set $G_k:=|f(\gamma_k)|_{(t_{l(k)}, t_{l(k)+1})}|$
belongs to $D^{\,\prime}\setminus E_*,$ because it does not contain
any point in $E_*.$ Since $D^{\,\prime}\setminus E_*$ consists only
of finite components, there exists at least one a component of
$D^{\,\prime}\setminus E_*,$ containing infinitely many components
of $G_k.$ Without loss of generality, going to a subsequence, if
need, we may assume that all $G_k$ belong to one component $K$ of
$D^{\,\prime}\setminus E_*.$

\medskip
Due to the compactness of $\overline{{\Bbb R}^n},$ we may assume
that the sequence $w_k:=f(\gamma_k(t_{l(k)})),$ $k=1,2,\ldots, $
converges to some a point $w_0\in \overline{D^{\,\prime}}.$ Let us
to show that $w_0\in C(f, \partial D)\subset E_*.$ Indeed, there are
two cases: either $w_k=f(\gamma_k(t_{l(k)}))\in E_*$ for infinitely
many $k,$ or $w_k\not\in E_*$ for infinitely many $k\in {\Bbb N}.$
In the first case, the inclusion $w_0\in E_*$ is obvious, because
$E_*$ is closed. In the second case, obviously, $w_k=f(x_k),$ but
this sequence converges to $y\in C(f,
\partial D)\subset E_*$ by the assumption.

By the assumption, each component of the set $D^{\,\prime}\setminus
E_*$ has a strongly accessible boundary with respect to
$\alpha$-modulus. In this case, for any neighborhood $U$ of the
point $w_0\in
\partial K$ there is a compact set $C_0^{\,\prime}\subset
K,$ a neighborhood $V$ of the point $w_0,$ $V\subset U,$ and a
number $P>0$ such that
\begin{equation}\label{eq1}
M_{\alpha}(\Gamma(C_0^{\,\prime}, F, K))\geqslant P
>0
\end{equation}
for any continua $F,$ intersecting $\partial U$ and $\partial V.$
Choose a neighborhood $U$ of $w_0$ with $d(U)< \delta/2(m-1),$ where
$\delta$ is from~(\ref{eq1B}). Let $C_0^{\,\prime}$ and $V$ be a
compact set and a neighborhood corresponding to $w_0.$

Observe that, $G_k$ contains some continuum $\widetilde{G}_k$ in $K$
intersecting $\partial U$ and $\partial V$ for sufficiently large
$k.$ Indeed, by the construction of $G_k,$ there is a sequence of
points $a_{s, k}:=f(\gamma_k(p_s))\rightarrow
w_k:=f(\gamma_k(t_{l(k)}))$ as $s\rightarrow \infty$ and $b_{s,
k}=f(\gamma_k(q_s))\rightarrow f(\gamma_k(t_{l(k)+1}))$ as
$s\rightarrow \infty,$ where $t_{l(k)})<
p_s<q_s<\gamma_k(t_{l(k)+1}))$ and $a_{s, k}, b_{s, k}\in K.$
By the triangle inequality and by~(\ref{eq7})
\begin{gather*}
|a_{s, k}-b_{s, k}|\geqslant |b_{s, k}-w_k|-|w_k -a_{s, k}|\nonumber \\
\label{eq11}
\geqslant|f(\gamma_k(t_{l(k)+1}))-w_k|-|f(\gamma_k(t_{l(k)+1}))-b_{s,
k}|-|w_k-a_{s, k}|\\
\geqslant\delta/(m-1)-|f(\gamma_k(t_{l(k)+1})-b_{s, k})|-|w_k-a_{s,
k}|\,.\nonumber
\end{gather*}
Since $a_{s, k}:=f(\gamma_k(p_s))\rightarrow
w_k:=f(\gamma_k(t_{l(k)}))$ as $s\rightarrow \infty$ and $b_{s,
k}=f(\gamma_k(q_s))\rightarrow f(\gamma_k(t_{l(k)+1}))$ as
$s\rightarrow \infty,$ it follows from the latter inequality that
there is $s=s(k)\in {\Bbb N}$ such that
\begin{equation}\label{eq12}
|a_{s(k), k}- b_{s(k), k}|\geqslant\delta/2(m-1)\,.
\end{equation}
Since $V$ is open, there is some neighborhood $U_k$ of $w_k$ such
$U_k\subset V.$ Since $a_{s, k}\rightarrow
w_k:=f(\gamma_k(t_{l(k)}))$ as $s\rightarrow \infty,$ we may assume
that $a_{s(k), k}\in V.$
Now, we set
$$\widetilde{G}_k:=f(\gamma_k)|_{[p_{s(k)}, q_{s(k)}]}\,.$$
In other words, $\widetilde{G}_k$ is a part of the path
$f(\gamma_k)$ between points $a_{s(k), k}$ and $b_{s(k), k}.$ Let us
to show that $\widetilde{G}_k$ intersects $\partial U$ and $\partial
V.$ Indeed, by the mentioned above, $a_{s(k), k}\subset V,$ so that
$V\cap \widetilde{G}_k\ne\varnothing.$ In particular, $U\cap
\widetilde{G}_k\ne\varnothing,$ because $V\subset U.$ On the other
hand, by~(\ref{eq12}) $d(\widetilde{G}_k)\geqslant\delta/2(m-1),$
however, $d(U)< \delta/2(m-1)$ by the choice of $U.$ In particular,
$d(V)<\delta/2(m-1).$ It follows from this that
$$(\overline{{\Bbb R}^n}\setminus U)\cap
\widetilde{G}_k\ne\varnothing\,, \qquad(\overline{{\Bbb
R}^n}\setminus V)\cap \widetilde{G}_k\ne\varnothing\,.$$
Now, by \cite[Theorem~1.I.5, \S46]{Ku}
$$\partial U\cap
\widetilde{G}_k\ne\varnothing\,, \partial V\cap
\widetilde{G}_k\ne\varnothing\,,$$
as required.
Now, by~(\ref{eq1}),
\begin{equation}\label{eq1C}
M_{\alpha}(\Gamma(\widetilde{G}_k, C_0^{\,\prime}, K))\geqslant P
>0\,, \qquad k=1,2,\ldots\,.
\end{equation}
Let us to show that, the relation~(\ref{eq1C}) contradicts the
definition of $f$ and condition~(\ref{eq6}). Indeed, we denote by
$\Gamma_k$ the family of all half-open paths $\beta_k:[a,
b)\rightarrow \overline{{\Bbb R}^n}$ such that $\beta_k(a)\in
|f(\gamma_k)|,$ $\beta_k(t)\in K$ for all $t\in [a, b)$ and,
moreover, $\lim\limits_{t\rightarrow b-0}\beta_k(t):=B_k\in
C_0^{\,\prime}.$ Obviously, by~(\ref{eq1C})
\begin{equation}\label{eq4A}
M_{\alpha}(\Gamma_k)=M_{\alpha}(\Gamma(\widetilde{G}_k,
C^{\,\prime}_0, K))\geqslant P
>0\,, \qquad k=1,2,\ldots\,.
\end{equation}
Since by the proving above $D\setminus E$ consists of finite number
of components, we may assume that all paths
$\nabla_k:=\gamma_k|_{[p_{s(k)}, q_{s(k)}]}$ belong to (some) one
component $D_1$ of $D\setminus E.$ By the proving above, $f$ is
closed in $D_1$ and $f(D_1)=K.$ Consider the family
$\Gamma_k^{\,\prime}$ of maximal $f$-liftings $\alpha_k:[a,
c)\rightarrow D$ of the family $\Gamma_k$ starting at $|\nabla_k|;$
such a family exists by Proposition~\ref{pr3}.

Note that, $C(f, \partial D_1)\subset C(f, \partial D) \cup C(f,
E)\subset E_*.$ Now, observe that, the situation when
$\alpha_k\rightarrow
\partial D_1$ as $k\rightarrow\infty$ is impossible because $f$ is
closed in $D_1$ and, consequently, $f$ is boundary preserving (see
\cite[Theorem~3.3]{Vu}). Therefore, by Proposition~\ref{pr3}
$\alpha_k\rightarrow x_1\in D$ as $t\rightarrow c-0,$ and $c=b$ and
$f(\alpha_k(b))=f(x_1).$ In other words, the $f$-lifting $\alpha_k$
is complete, i.e., $\alpha_k:[a, b]\rightarrow D_1.$ Besides that,
it follows from that $\alpha_k(b)\in f^{\,-1}(C^{\,\prime}_0)\cap
D_1.$ Since $f(D_1)=K,$ we obtain that
$f(f^{\,-1}(C^{\,\prime}_0)\cap D_1)=C^{\,\prime}_0.$ In addition,
$f^{\,-1}(C^{\,\prime}_0)\cap D_1$ is a compactum in $D_1,$ see e.g.
\cite[Theorem~3.3]{Vu}. Let $\varepsilon_1:={\rm
dist}\,((f^{\,-1}(C^{\,\prime}_0)\cap D_1), x_0).$ We may consider
that $C_2:=f^{\,-1}(C^{\,\prime}_0)\cap D_1\subset D_1\setminus
d_1.$ By the construction, $|\nabla_k|\subset D_1\cap d_k.$

\medskip
Applying now Lemmas~\ref{thOS4.1} and \ref{lem4A}, we obtain for
$\alpha:=\frac{p}{p-n+1}$ and sufficiently large $k\in {\Bbb N}$
that
\begin{gather}
M_{\alpha}(f(\Gamma_k^{\,\prime}))=M_{\alpha}(\Gamma(\widetilde{G}_k,
C^{\,\prime}_0,
K))\nonumber \\
\label{eq15} \leqslant (N(f,
D_1))^{\frac{n-1}{p-n+1}}\int\limits_{A(x_0, r_k,
\varepsilon_1)}Q(x)\eta^{\,\alpha}(|x-x_0|)\,dm(x)
\end{gather}
for any nonnegative Lebesgue measurable function $\eta$ satisfying
the relation~(\ref{eq6B}). Now, we set
$$\widetilde{Q}(x)=\begin{cases}Q(x), & Q(x)\geqslant 1\\
 1, & Q(x)<1\end{cases}\,.$$
Besides that, we set
\begin{equation}\label{eq9A} I_k=I(x_0, r_k,
\varepsilon_1)=\int\limits_{r_k}^{\varepsilon_1}\
\frac{dr}{r^{\frac{n-1}{\alpha-1}}\widetilde{q}_{x_0}^{\frac{1}{\alpha-1}}(r)}\,,
\end{equation}
where
$$\widetilde{q}_{x_0}(r)=\frac{1}{\omega_{n-1}r^{n-1}} \int\limits_{S(x_0,
r)}\widetilde{Q}(x)\,d\mathcal{H}^{n-1}\,.$$
By the assumption,
\begin{equation}\label{eq2_2}
I_k=\int\limits_{r_k}^{\varepsilon_1}\
\frac{dr}{r^{\frac{n-1}{\alpha-1}}\widetilde{q}_{x_0}^{\frac{1}{\alpha-1}}(r)}\rightarrow\infty
\end{equation}
as $k\rightarrow\infty.$ Since $\widetilde{q}_{x_0}(r)\geqslant 1$
for a.e. $r$ and by~(\ref{eq2_2}), $0<I_k<\infty$ for sufficiently
large $k\in {\Bbb N}.$ We set
$$\psi_k(t)= \left \{\begin{array}{rr}
1/[t^{\frac{n-1}{\alpha-1}}\widetilde{q}_{x_0}^{\frac{1}{\alpha-1}}(t)],
& t\in (r_k, \varepsilon_1)\ ,
\\ 0,  &  t\notin (r_k, \varepsilon_1)\ .
\end{array} \right. $$
Let $\eta_k(t)=\psi_k(t)/I_k$ for $t\in (r_k, \varepsilon_1)$ and
$\eta_k(t)=0$ otherwise. Now, $\eta_k$ satisfies~(\ref{eq6B}). In
this case, by~(\ref{eq15}) and by Fubini's theorem,
\begin{equation}\label{eq11*}
M_{\alpha}(f(\Gamma^{\,\prime}_k))=M_{\alpha}(\Gamma(\widetilde{G}_k,
C^{\,\prime}_0, K))=\frac{\omega_{n-1}}{I^{\alpha-1}_k}\rightarrow 0
\end{equation}
as $k\rightarrow\infty.$ The relations~(\ref{eq11*})
and~(\ref{eq1C}) contradict each other. Theorem is proved, excluding
the equality $f(\overline{D}_P)=\overline{D^{\,\prime}}.$ The proof
of the latter repeats the last part of the proof of Theorem~5.5 in
\cite{Sev$_4$}.~$\Box$

\medskip
The following lemma is proved in \cite[Lemma~2.1]{Sev$_5$}.

\medskip
\begin{lemma}\label{lem1}
{\it\, Let $1\leqslant p\leqslant n,$ and let $\Phi:[0,
\infty]\rightarrow [0, \infty] $ be a strictly increasing convex
function such that the relation
$$
\frac{d\tau}{\tau\left[\Phi^{-1}(\tau)\right]^{\frac{1}{p-1}}}=
\infty
$$
holds for some $\delta_0>\tau_0:=\Phi(0).$ Let $\frak{Q}$ be a
family of functions $Q:{\Bbb R}^n\rightarrow [0, \infty]$ such that
$$
\int\limits_D\Phi(Q(x))\frac{dm(x)}{\left(1+|x|^2\right)^n}\
\leqslant M_0<\infty
$$
for some $0<M_0<\infty.$ Now, for any $0<r_0<1$ and for every
$\sigma>0$ there exists $0<r_*=r_*(\sigma, r_0, \Phi)<r_0$ such that
$$
\int\limits_{\varepsilon}^{r_0}\frac{dt}{t^{\frac{n-1}{p-1}}q^{\frac{1}{p-1}}_{x_0}(t)}\geqslant
\sigma\,,\qquad \varepsilon\in (0, r_*)\,,
$$
for any $Q\in \frak{Q}.$ }
\end{lemma}

\medskip
{\it Proof of Theorem~\ref{th3_1}}. By Lemma~\ref{lem1} and
Remark~\ref{rem1} the conditions~(\ref{eq1D})--(\ref{eq2}) imply the
relation~(\ref{eq6}) for every $b\in \partial D.$ The rest follows
by Theorem~\ref{th3}.~$\Box$

\medskip
{\it Proof of Theorem~\ref{th1}} is similar to the proof of
Theorem~\ref{th3}, but uses Lemma~\ref{lem3} instead of
Lemma~\ref{thOS4.1}. In all other respects, the proof of this
theorem literally repeats the previous proof.~$\Box$

\medskip
{\it Proof of Theorem~\ref{th1_1}}. By Lemma~\ref{lem1} and
Remark~\ref{rem1} the conditions~(\ref{eq1D})--(\ref{eq2}) imply the
relation~(\ref{eq6}) for every $b\in \partial D.$ The rest follows
by Theorem~\ref{th1}.~$\Box$

\medskip {\bf Acknowledgement}

The work was supported by the National Research Foundation of
Ukraine (Project ``Analogues of Carath\'{e}odory and Koebe-Bloch
theorems for Orlycz-Sobolev classes'', Project number 2025.02/0010).

\medskip
\noindent{{\bf Zarina Kovba} \\
Zhytomyr Ivan Franko State University,  \\
40 Velyka Berdychivs'ka Str., 10 008  Zhytomyr, UKRAINE \\
e-mail: mazhydova@gmail.com

\medskip
{\bf \noindent Evgeny Sevost'yanov} \\
{\bf 1.} Zhytomyr Ivan Franko State University,  \\
40 Velyka Berdychivs'ka Str., 10 008  Zhytomyr, UKRAINE \\
{\bf 2.} Institute of Applied Mathematics and Mechanics\\
of NAS of Ukraine, \\
19 Henerala Batyuka Str., 84 116 Slov'yansk,  UKRAINE\\
esevostyanov2009@gmail.com


\begin{thebibliography}{99}

\bibitem[Ad]{Ad} {\sc Adamowicz, T.:} Prime ends in metric spaces and quasiconformal-type
mappings. - Analysis and Mathematical Physics~9:4, 2019, 1941-–1975.

\bibitem[Car]{Car} {\sc Caratheodory, C.:} \"Uber die Begrenzung
der einfachzusammenh\"angender Gebiete. - Math. Ann.~73, 1913,
323--370.

\bibitem[CL]{CL} {\sc Collingwood, E.F. and A.J.~Lohwator}: The Theory of Cluster
Sets. -- Cambridge Tracts in Math. and Math. Physics~\textbf{56},
Cambridge Univ. Press, Cambridge, 1966.

\bibitem[Cr]{Cr} {\sc Cristea, M.:} Open discrete mappings having local $ACL^n$
inverses. - Complex Variables and Elliptic Equations~55: 1--3, 2010,
61--90.

\bibitem[DS]{DS} {\sc Desyatka, V., E.~Sevost'yanov:}
On boundary-non-preserving mappings with Poletsky inequality. -
Canadian Mathematical Bulletin~68:3, 2025, 834--855.

\bibitem[DSH]{DSH} {\sc Desyatka, V., E.~Sevost'yanov,
A.~Halyts'ka:} On boundary extension of unclosed Orlycz-Sobolev
mappings, https://arxiv.org/abs/2504.18123 .

\bibitem[Hes]{Hes} {\sc Hesse, J.:}
A $p-$extremal length and $p-$capacity equality. - Ark. Mat.~13,
1975, 131--144.

\bibitem[IS]{IS} {\sc Ilyutko, D.P., E.A.~Sevost'yanov:} On prime ends on Riemannian
manifolds. - J. Math. Sci.~241:1, 2019. 47--63.

\bibitem[KPR]{KPR} {\sc Kovtonyuk, D.A., I.V.~Petkov and V.I.~Ryazanov:}
Prime ends in theory of mappings with finite distortion in the
plane. - Filomat J.~31:5, 2017, 1349–-1366.

\bibitem[KR]{KR} {\sc Kovtonyuk, D.A. and V.I.~Ryazanov:}
On the theory of prime ends for space mappings. - Ukr. Math.
J.~67:4, 2015, 528--541.

\bibitem[Ku]{Ku} {\sc Kuratowski, K.:} Topology, v.~2. -- Academic
Press, New York--London,  1968.

\bibitem[MRV]{MRV} {\sc Martio, O., S. Rickman, and J. V\"{a}is\"{a}l\"{a}:}
Topological and metric properties of quasiregular mappings. - Ann.
Acad. Sci. Fenn. Ser. A1. 488, 1971, 1--31.

\bibitem[MRSY]{MRSY} {\sc Martio, O., V. Ryazanov, U. Srebro, and E. Yakubov:}
Moduli in modern mapping theory. - Springer Science + Business
Media, LLC, New York, 2009.

\bibitem[Na]{Na} {\sc N\"{a}kki, R.:} Prime ends and quasiconformal
mappings. - J. Anal. Math.~35, 1979, 13--40.

\bibitem[PSS]{PSS} {\sc Petkov, I., R.~Salimov, M.~Stefanchuk,}
Nonlinear Beltrami equation: Lower estimates of Schwarz lemma's
type. - 67:3, 2024, 533–-543.

\bibitem[RV]{RV} {\sc Ryazanov, V., S.~Volkov:}
On the Boundary Behavior of Mappings in the Class $W^{1,1}_{loc}$ on
Riemann Surfaces. - Complex Analysis and Operator Theory~11, 2017,
1503--1520.

\bibitem[Sal]{Sal} {\sc Salimov, R.R.:} Estimation of the measure of the image of the ball. - Sib.
Math. J.~53:4, 2012, 739--747.

\bibitem[SalSt]{SalSt} {\sc Salimov, R.R., M.V.~Stefanchuk,}
Functional Asymptotics of Solutions of the Nonlinear
Cauchy–Riemann–Beltrami System. - Journal of Mathematical Sciences
(United States)~277:2, 2023, 311--328.

\bibitem[Sev$_1$]{Sev$_1$} {\sc Sevost'yanov, E.:}
On the local behavior of Open Discrete Mappings from the
Orlicz-Sobolev Classes. - Ukr. Math. J.~68:9, 2017, 1447--1465.

\bibitem[Sev$_2$]{Sev$_2$} {\sc Sevost'yanov, E.:}
On the boundary behavior of some classes of mappings. - J. Math.
Sci.~243:6, 2019, 934--948.

\bibitem[Sev$_3$]{Sev$_3$} {\sc Sevost'yanov, E.:}
The inverse Poletsky inequality in one class of mappings. - Journal
of Mathematical Sciences~264:4, 2022, 455--470.

\bibitem[Sev$_4$]{Sev$_4$} {\sc Sevost'yanov, E.A:}
Mappings with Direct and Inverse Poletsky Inequalities. Developments
in Mathematics (DEVM, volume 78). - Springer Nature Switzerland AG,
Cham, 2023.

\bibitem[Sev$_5$]{Sev$_5$} {\sc Sevost'yanov, E.:} On global behavior of mappings with integral
constraints. - Analysis and Mathematical Physics~12:3, 2022, Article
number~76.

\bibitem[Shl]{Shl} {\sc Shlyk, V.A.:} The equality between $p$-capacity and
$p$-modulus. - Siberian Mathematical Journal~34:6, 1993, 1196--1200.

\bibitem[Va]{Va} {\sc V\"{a}is\"{a}l\"{a} J.:} Lectures on
$n$-dimensional quasiconformal mappings. - Lecture Notes in Math.
229, Springer-Verlag, Berlin etc., 1971.

\bibitem[Vu]{Vu} {\sc Vuorinen, M.:} Exceptional sets and boundary behavior
of quasiregular mappings in $n$-space. - Ann. Acad. Sci. Fenn. Ser.
A 1. Math. Dissertationes~11, 1976, 1--44.

\bibitem[Zi$_1$]{Zi$_1$} {\sc Ziemer, W.P.:} Extremal length and conformal capacity. - Trans. Amer.
Math. Soc.~126:3, 1967, 460--473.

\bibitem[Zi$_2$]{Zi$_2$} {\sc Ziemer, W.P.:} Extremal length and $p$-capacity. -
Michigan Math. J.~16, 1969, 43--51.


\end{thebibliography}
\end{document}